\documentclass[12pt]{amsart}
\usepackage{amsmath,amssymb,enumerate,mathtools}
\newtheorem{theorem}{Theorem}[section]
\newtheorem{claim}{Claim}[theorem]
\newtheorem{lemma}[theorem]{Lemma}

\newtheorem{conjecture}[theorem]{Conjecture}
\theoremstyle{definition}

\newcommand{\bF}{\mathbb F}

\newcommand{\bZ}{\mathbb Z}

\newcommand{\cB}{\mathcal{B}}

\newcommand{\cF}{\mathcal{F}}

\newcommand{\cM}{\mathcal{M}}

\newcommand{\dcon}{/ \!\! /}
\DeclareMathOperator{\si}{si}

\DeclareMathOperator{\cl}{cl}

\DeclareMathOperator{\dist}{dist}

\DeclareMathOperator{\PG}{PG}

\DeclareMathOperator{\GF}{GF}

\newcommand{\del}{\!\setminus\!}
\newcommand{\con}{/}
\newcommand{\wh}{\widehat}

\begin{document}
\title[Spanning Cliques and Geometries]{The Structure of Matroids with a Spanning Clique or Projective Geometry}
\thanks{This research is partly supported by a Discovery Grant [203110-2011] from the Natural Sciences and Engineering Research Council
of Canada and a Research Grant [N00014-10-1-0851] from the Office of Naval Research.}
\author[Geelen]{Jim Geelen}
\author[Nelson]{Peter Nelson}
\address{Department of Combinatorics and Optimization, University of Waterloo, Canada}
\begin{abstract}
	Let $s,n \ge 2$ be integers. We give a qualitative structural description of every matroid $M$ that is spanned by a frame matroid of a complete graph and has no $U_{s,2s}$-minor and no rank-$n$ projective geometry minor, showing that every such matroid is `close' to a frame matroid. We also give a similar description of every matroid $M$ with a spanning projective geometry over a field $\GF(q)$ as a restriction and with no $U_{s,2s}$-minor and no $\PG(n,q')$-minor for any $q' > q$, showing that such an $M$ is `close' to a $\GF(q)$-representable matroid. 
\end{abstract}
\maketitle	

\section{Introduction}

In~[\ref{highlyconnected}], Geelen, Gerards, and Whittle describe the structure of 
highly-connected matroids in minor-closed classes of matroids represented over a fixed finite field.
In the same paper they conjecture extensions of their results to minor-closed classes
of matroids omitting a fixed uniform minor. The main results in this paper are motivated
by those conjectures, which we shall restate at the end of this introduction.
Here we are primarily concerned with the structure of matroids having either the cycle-matroid
of a complete graph or a projective geometry as spanning restriction.

An \emph{elementary projection} of a matroid $M$ is a matroid obtained from an extension of $M$ by contracting the new element, and an \emph{elementary lift} of $M$ is one obtained from a coextension by deleting the new element. Given two matroids $M$ and $N$ on the same ground set, we say that $N$ is a {\em distance-$k$ perturbation} of $M$ if $N$ can be obtained from 
$M$ by a sequence of $k$ elementary lifts and elementary projections. Perturbations play a natural role in considering
minor-closed classes of matroids omiting a uniform matroid. In particular, if $\cM$ is a minor-closed class of matroids that omits a uniform matroid, then the set of matroids  that are distance-$k$ perturbations of matroids in $\cM$ is also minor-closed and omits a uniform matroid; see Theorem~\ref{perturbthm}. Note that the uniform matroid $U_{r,n}$ is contained as a minor of $U_{s,2s}$ where $s=\max(r,n-r)$, so it suffices to consider classes omitting `balanced' uniform matroids 
$U_{s,2s}$. 

We start with the easier of our two main results which concerns matroids with a spanning projective geometry restriction.
\begin{theorem}\label{main2}
	For all integers $s,n \ge 2$, there exists an integer $k$ such that, for every prime power $q$ and every rank-$r$ matroid $M$ with a $\PG(r-1,q)$-restriction,  either $M$ has a $U_{s,2s}$-minor, $M$ has a $\PG(n-1,q')$-minor for some $q' > q$, or there is a distance-$k$ perturbation of $M$ that is $\GF(q)$-representable.
\end{theorem}

A matroid $M$ is \emph{framed by $B$} if $B$ is a basis of $M$ and each element of $M$ is spanned by a subset of $B$ with at most two elements. A \emph{$B$-clique} is a matroid framed by $B$ so that each pair of distinct elements in $B$ is contained in a triangle. The second of our main results concerns matroids with a spanning $B$-clique restriction.
\begin{theorem}\label{main1}
	For all integers $s,n \ge 2$, there exists an integer $k$ such that, if $M$ is a matroid with a spanning $B$-clique restriction, then either $M$ has a $U_{s,2s}$-minor, $M$ has a rank-$n$ projective geometry minor, or there is a distance-$k$ perturbation of $M$ that is framed by $B$. 
\end{theorem}

Theorem~\ref{main1} has an interesting special case where the spanning clique is `bicircular'; in this case we can avoid 
the outcome giving a large projective geometry as a minor. Given a graph $G = (V,E)$, we write $B^+(G)$ for the \emph{framed bicircular matroid} of $G$; this is the matroid with ground set $E \cup V$, in which a set $X$ is independent if and only if $|X \cap (E(H) \cup V(H))| \le |V(H)|$ for each subgraph $H$ of $G$. Equivalently,  $B^+(G)$ is constructed from the free matroid on $V$ by adding each $e = v_1v_2 \in E$ freely to the line between the basis elements $v_1$ and $v_2$. 
Note that $B^+(K_n)$ is a $V(K_n)$-clique. The \emph{bicircular matroid} of $G$, in the more usual sense, is just the matroid $B(G) = B^+(G) \del V$. As a corollary of Theorems~\ref{main1} and~\ref{main2}, we get the following strengthening of Theorem~\ref{main1}. 

\begin{theorem}\label{bicirc}
		For every integer $s \ge 2$ there is an integer $k$ such that, if $M$ is a rank-$r$ matroid with no $U_{s,2s}$-minor and with a $B^+(K_r)$-restriction framed by $B$,  then there is a distance-$k$ perturbation of $M$ that is framed by $B$.
\end{theorem}
	
	In Section~\ref{selfdualsection}, we prove a result of independent interest, Theorem~\ref{selfdual}, that finds the unavoidable minors for arbitrary large matroids that have two disjoint bases. A corollary is the following, which finds one of two specific minors in any matroid that is not close to being `trivial'. 
	
	\begin{theorem}\label{unavoidable}
		Let $s \ge 0$ be an integer and $k = 4^{4^{2s^2}}$. Then, for each matroid  $M$, either
		\begin{itemize}
			\item $M$ has a $U_{s,2s}$-minor, 
			\item $M$ has a minor isomorphic to the direct sum of $s$ copies of $U_{1,2}$, or 
			\item there is a distance-$k$ perturbation of $M$ whose elements are all loops or coloops. 
		\end{itemize}
	\end{theorem}

	\subsection*{Structure Theory}
	Theorems~\ref{main1} and~\ref{main2} fit into a larger, mostly conjectural, regime of structure theory in minor-closed classes omitting a uniform matroid. The first of these conjectures predicts the  unavoidable minors for very highly connected matroids. A matroid is \emph{vertically $k$-connected} if, for every $A \subseteq E(M)$ with $\lambda_M(A) < k-1$, either $A$ or $E(M)-A$ is spanning in $M$. The following conjecture was posed in~[\ref{highlyconnected}].	
	\begin{conjecture}\label{highconn}
		For all $n \ge 2$ there is an integer $k$ such that, if $M$ is a vertically $k$-connected matroid with $|M|\ge 2k$, then $M$ or $M^*$ has a minor isomorphic to one of $M(K_n),B(K_n),$ or $U_{n,2n}$. 
	\end{conjecture}
	
	While $M(K_n)^*$ and $B(K_n)^*$ are not  even  vertically $4$-connected themselves, they do contain minors with high vertical connectivity; indeed, for each $k$ there is a graph $G$ so that $M(G)^*$ and $B(G)^*$ are both vertically $k$-connected. 
	To obtain such a graph one can take  a $k$-regular Cayley graph with girth at least $k$ (see Margulis~[\ref{margulis}]
	for the construction); by~[\ref{gr}, Theorem 3.4.2], these graphs are $k$-connected.
	
	 In any case, the dual outcomes in Conjecture~\ref{highconn} are perhaps not needed if $M$ has large co-rank.
	
	\begin{conjecture}
		For all $n \ge 2$ there is an integer $k$ so that, if $M$ is a vertically $k$-connected matroid with $|M|\ge 2k$ and $r(M^*) \ge r(M)$, then $M$ has a minor isomorphic to one of $M(K_n)$, $B(K_n)$ or $U_{n,2n}$. 
	\end{conjecture}
	
	The following conjecture, which is essentially posed in~[\ref{highlyconnected}], states that any highly vertically connected matroid omitting a given uniform minor is close to having one of three specific structures that preclude such a minor. Here a \emph{frame matroid} is one of the form $M \del B$, where $M$ is a matroid framed by $B$; the matroid $U_{4,8}$ is not frame.
	\begin{conjecture}
		For all $s \ge 4$ there is an integer $k$ so that, if $M$ is a vertically $k$-connected matroid with no $U_{s,2s}$-minor, then there is a distance-$k$ perturbation $N$ of $M$ such that either
				\begin{itemize}
			\item $N$ or $N^*$ is a frame matroid, or
			\item $N$ is $\bF$-representable for some field $\bF$ over which $U_{s,2s}$ is not representable. 
		\end{itemize}
	\end{conjecture}
	
	We assume familiarity with matroid theory, using the notation of Oxley [\ref{oxley}].

\section{Covering Number} Let $a \ge 1$ be an integer. We write $\tau_a(M)$ for the \emph{$a$-covering number} of a matroid $M$, defined to be the minimum number of sets of rank at most $a$ in $M$ required to cover $E(M)$. For $a \ge 2$, the parameter $\tau_{a-1}$ is a useful measure of density when excluding a rank-$a$ uniform minor; the following lemma, a strengthening of one proved in [\ref{gkep}], finds such a minor when $\tau_{a-1}$ is large enough compared to the rank.  

\begin{lemma}\label{udensity}
	Let $a,b$ be integers with $1 \le a < b$. If $B$ is a basis of a matroid $M$ satisfying $r(M) > a$ and $\tau_{a}(M) \ge \binom{b}{a}^{r(M)-a}$,  then $M$ has a $U_{a+1,b}$-minor $U$ in which $E(U) \cap B$ is a basis.
\end{lemma}
\begin{proof}
	If $r(M) = a+1$, then note that $M|B \cong U_{a+1,a+1}$; let $X \subseteq E(M)$ be maximal so that $B \subseteq X$ and $M|X$ is a rank-$(a+1)$ uniform matroid. We may assume that $|X| < b$. By maximality, every $x \in E(M) - X$ is spanned by some $a$-element subset of $X$; this also holds for every $x \in X$, so $\tau_{a}(M) \le \binom{|X|}{a} < \binom{b}{a}$, a contradiction. 
	
	Let $r(M) = r > a+1$ and suppose that the lemma holds for matroids of smaller rank. Let $e \in B$. We may assume that $M \con e$ has no $U_{a+1,b}$-minor $U$ in which $E(U) \cap (B-\{e\})$ is a basis, so $\tau_{a}(M \con e) \le \binom{b}{a}^{r-a-1}$. Let $\cF$ be a cover of $M \con e$ with $\binom{b}{a}^{r-a-1}$ rank-$a$ sets. Since $\binom{b}{a}^{r-a} \le \tau_{a}(M) \le \sum_{F \in \cF}\tau_{a}(M|(F \cup e))$, there is some $F \in \cF$ such that $\tau_{a}(M|(F \cup \{e\})) \ge \binom{b}{a}$. Note that $r_M(F \cup \{e\}) = a+1$. By contracting a maximal subset of $B$ that is skew to $F \cup \{e\}$ in $M$, we obtain a rank-$(a+1)$ minor $N$ of $M$ so that $E(N) \cap B$ is a basis of $N$, and $M \con (F \cup \{e\})$ is a restriction of $N$. Now $\tau_{a}(N) \ge \tau_a(M|(F \cup \{e\})) \ge \binom{b}{a}$, and the lemma follows by the inductive hypothesis.
\end{proof}

The next two results, which find a projective geometry minor or large uniform minor whenever the covering number is large, are special cases of the main theorems of [\ref{covering1}] and [\ref{covering2}] respectively. 

\begin{theorem}\label{pgdensity}
	There is a function $f_{\ref{pgdensity}}\colon \bZ^2 \to \bZ$ so that, for all integers $s,n \ge 2$, if $M$ is a matroid with $r(M) > 1$ and $\tau_{s-1}(M) \ge r(M)^{f_{\ref{pgdensity}}(s,n)}$, then $M$ has a $U_{s,2s}$-minor or a rank-$n$ projective geometry minor. 
\end{theorem}

\begin{theorem}\label{exppgdensity}
	There is a function $f_{\ref{exppgdensity}}\colon \bZ^3 \to \bZ$ so that, for all integers $s,n \ge 2$, and every prime power $q$, if $M$ is a matroid with $\tau_{s-1}(M) \ge q^{r(M) + f_{\ref{exppgdensity}}(s,q,n)}$, then $M$ has a $U_{s,2s}$-minor or a $\PG(n-1,q')$-minor for some $q' > q$. 
\end{theorem}

\section{Disjoint Bases}\label{selfdualsection}

For $t \ge 0$, let $tU_{1,2}$ denote the direct sum of $t$ copies of $U_{1,2}$. Both $tU_{1,2}$ and $U_{s,2s}$   are the union of two disjoint bases. In this section we show that any  large matroid with two disjoint bases has one of these two as a minor. 

\begin{lemma}\label{goodbasis}
	Let $M$ be a rank-$r$ matroid with ground set $\{x_1, \dotsc, x_n\}$. There exists $t \in \{0,\dotsc,n\}$ so that the set $X = \{x_1, \dotsc, x_t\}$ satisfies $|\cl_M(X) \cap (E-X)| \ge  \lfloor\tfrac{n}{r+1}\rfloor$. 
\end{lemma}
\begin{proof}
For each $t\in \{0,\ldots,n\}$ let $a_t=r_M(\{x_1,\ldots,x_t\})$. Thus
$0=a_0\le a_1\le\cdots\le a_n = r$. So there exist $t,t'\in\{0,\ldots,n\}$ such that $a_t=a_t'$ and $t'-t\ge \lceil\tfrac{n+1}{r+1}\rceil>\lfloor\tfrac{n}{r+1}\rfloor$. Now let $X=\{x_1,\ldots,x_t\}$ and observe that $x_{t+1},\ldots,x_{t'}\in \cl_M(X)$.
\end{proof}

Let $\bar{A} = (a_1, \dotsc, a_n)$ and $\bar{B} = (b_1, \dotsc, b_n)$ be $n$-tuples of distinct elements of a matroid $M$. We say that the pair $(\bar{A},\bar{B})$ is \emph{upper-triangular} in $M$ if $(\{a_1, \dotsc, a_n\},\{b_1, \dotsc, b_n\})$ is a partition of $E(M)$ into two bases, and $\cl_M(\{a_1, \dotsc, a_k\}) = \cl_M(\{b_1, \dotsc, b_k\})$ for every $k \in \{1, \dotsc, n\}$. The pair is \emph{lower-triangular} if it is an upper-triangular pair of $M^*$; it is easy to check that this is equivalent to the condition that $\cl_M(\{a_k,\dotsc,a_n\}) = \cl_M(\{b_k, \dotsc, b_n\})$ for each $k$. This terminology is motivated by matrix representations; $(\bar{A},\bar{B})$ is an upper-triangular pair in a representable matroid $M$ if $M = M(P)$, where $P$ is a matrix with column set $\bar{A} \cup \bar{B}$ for which $P[\bar{A}]$ is an identity matrix and $P[\bar{B}]$ is an upper-triangular matrix.

Given an upper-triangular pair $((a_1, \dotsc, a_n),(b_1,\dotsc,b_n))$ in $M$, there is a natural involution $\varphi\colon E(M) \to E(M)$ defined by $\varphi(a_i) = b_i$ and $\varphi(b_i) = a_i$. The next lemma shows that contracting/deleting pairs of sets related by $\varphi$ does not destroy upper-triangularity. In what follows we  mix set notation and tuple notation where there is no ambiguity.

\begin{lemma}\label{staytriangular}
	Let $(\bar{A},\bar{B})$ be an upper-triangular pair in a matroid $M$ and let $\varphi\colon E(M) \to E(M)$ be the associated involution. Let $X \subseteq \bar{A}$ and $Y = \varphi(X)$. Then $(\bar{A}-X,\bar{B}-Y)$ is upper-triangular in both $M \con X \del Y$ and $M \con Y \del X$. 
\end{lemma}
\begin{proof}
	Let $\bar{A} = (a_1, \dotsc, a_n)$ and $\bar{B} = (b_1, \dotsc, b_n)$. Inductively and by symmetry, it suffices to show for each $\ell$ that $(\bar{A}-\{a_\ell\},\bar{B}-\{b_{\ell}\})$ is upper-triangular in $N = M \con a_\ell \del b_\ell$. Indeed, for each $k < \ell$ we have $\cl_N(\{a_1,\dotsc, a_k\}) \supseteq \cl_M(\{a_1, \dotsc, a_k\}) \supseteq \{b_1, \dotsc, b_k\}$ and vice versa, so 
$\cl_N(\{a_1, \dotsc, a_k\}) = \cl_N(\{b_1, \dotsc, b_k\})$.
 For $k > \ell$,
 \begin{align*}
 	\cl_N(\{a_1,\dotsc, a_k\} - \{a_\ell\}) &= \cl_M(\{a_1, \dotsc, a_k\}) - \{a_{\ell},b_{\ell}\} \\& \supseteq \{b_1, \dotsc, b_k\} - \{b_\ell\}, \text{   and}\\
	\cl_N(\{b_1, \dotsc, b_k\} - \{b_\ell\}) &= \cl_{M \con a_\ell}(\{b_1, \dotsc, b_{\ell-1}\} \cup \{b_{\ell+1}, \dotsc, b_k\}) - \{b_\ell\} \\
	&= \cl_M(\{a_1, \dotsc, a_{\ell-1},a_\ell\} \cup \{b_\ell+1, \dotsc, b_k\}) - \{a_\ell,b_\ell\} \\
	&= \cl_M(\{b_1, \dotsc, b_k\}) - \{a_{\ell},b_\ell\}\\
	&\supseteq \{a_1,\dotsc, a_k\} - \{a_\ell\},
\end{align*} where we use the fact that $\cl_M(\{a_1, \dotsc, a_\ell\}) = \cl_M(\{b_1, \dotsc, b_\ell\})$. Thus $\cl_N(\{a_1, \dotsc, a_k\}-\{a_\ell\}) = \cl_N(\{b_1, \dotsc, b_k\} - \{b_\ell\})$. The lemma follows. 
\end{proof}

\begin{lemma}\label{triangularone}
	Let $s \ge 2$ and $t \ge 0$ be integers. If $A$ and $B$ are disjoint bases of a matroid $M$ with $r(M) \ge 4^{st}$, then either 
	\begin{itemize}
		\item $M$ has a $U_{s,2s}$-minor $U$ in which $E(U) \cap A$ and $E(U) \cap B$ are bases, or
		\item $M$ has a rank-$t$ minor with a lower-triangular pair $(\bar{A},\bar{B}) \in A^t \times B^t$.	
\end{itemize}
\end{lemma}
\begin{proof}
	We may assume that $E(M) = A \cup B$, that the first outcome does not hold, and inductively that $t \ge 1$ and the lemma holds for smaller $t$. Let $A_0$ be an $s$-element subset of $A$, and let $A' = A- A_0$. Let $B_0$ be a basis for $M \con A' \del A_0$. Since $r(M \con A') = s$, by Lemma~\ref{udensity} we have $\tau_{s-1}(M \con A') \le \binom{2s}{s-1}$, so by a majority argument there is some $B'' \subseteq B-B_0$ for which $r_{M \con A'}(B'') \le s-1$ and $|B''| \ge (n-s)/\binom{2s}{s-1} \ge 4^{-s}n + s$ (this last inequality follows from $n \ge 4s$ and $\binom{2s}{s-1} \le \tfrac{1}{2}4^s$). Let $Y$ be a basis for $B''$ in $M \con A'$ and let $B_1 = B''-Y$, so $|B_1| \ge 4^{-s}n \ge 4^{s(t-1)}$. 	Let $M' = (M \con Y)|(A \cup B_0 \cup B_1).$
	Note that, since they are bases for $M$, both $A$ and $A' \cup B_0$ are spanning in $M'$, and $r(M') = n - |Y|$. Moreover, $A'$ is independent in $M'$ and $B_1 \subseteq \cl_{M'}(A')$. Thus,  
	\begin{align*}
		\sqcap_{(M')^*}(A_0,B_0) &= r(M' \del A_0) + r(M' \del B_0) - r(M' \del (A_0 \cup B_0)) - r(M') \\
		&\ge r_{M'}(A' \cup B_0) + r_{M'}(A) - r_{M'}(A' \cup B_1) - r(M')\\ 
		&= r(M') - r_{M'}(A')\\
		&= (n-|Y|) - (n-s) > 0,
	\end{align*}
	so $M'$ has a cocircuit $K \subseteq A_0 \cup B_0$ that intersects both $A_0$ and $B_0$. Let $a \in K \cap A_0$ and $b \in K \cap B_0$. Note that $B_1$ is independent in $M'$ and is spanned by $A'$; let $A_1 \subseteq A'$ be such that $|A_1| = |B_1|$ and $B_1$ spans $A_1$ in $M' \con (A'-A_1)$. Let
	\[M'' = M' \con (A'-A_1) \con (A_0 \cup B_0 - \{a,b\}).\] 
	Now $\{a,b\} \in A_0 \times B_0$ is a series pair of $M''$, and $A_1 \subseteq A'$ and $B_1 \subseteq B- B_0$ are bases of $M'' \con a \del b$. Since $r(M'' \con a \del b) = |B''| \ge 4^{s(t-1)}$, by the inductive hypothesis there is a rank-$(t-1)$ minor $N_0 = M'' \con (\{a\} \cup C) \del (\{b\} \cup D)$ of $M'' \con a \del b$ having a lower-triangular pair $(\bar{A_0}, \bar{B_0})$ with $\bar{A}_0 \subseteq A$ and $\bar{B}_0 \subseteq B$. Now $N = M'' \con C \del D$ has $\{a,b\}$ as a series pair and $N \con a \del b = N_0$. It follows that $((a,\bar{A}_0),(b,\bar{B}_0))$ is a lower-triangular pair in the rank-$t$ matroid $N$, giving the result. 
\end{proof}

\begin{lemma}\label{triangulartwo}
	Let $s \ge 2$ and $t \ge 0$ be integers. If $M$ is a matroid with $r(M) \ge (s4^s)^t$, and $(\bar{A},\bar{B})$ is an upper-triangular pair of $M$, then either 
	\begin{itemize}
		\item $M$ has a $U_{s,2s}$-minor $U$ in which $E(U) \cap \bar{A}$ and $E(U) \cap \bar{B}$ are bases, or
		\item $M$ has a $tU_{1,2}$-minor $N$ in which $E(N) \cap \bar{A}$ and $E(N) \cap \bar{B}$ are bases. 
	\end{itemize}
\end{lemma}
\begin{proof}
	We may assume that the first outcome does not hold, and inductively that $t \ge 1$ and the lemma holds for smaller $t$. Let $n = r(M)$, let $\bar{A} = (a_1, \dotsc, a_n)$, $\bar{B} = (b_1, \dotsc, b_n)$ and let $\varphi$ be the involution associated with $(\bar{A},\bar{B})$. Since $(\bar{A},\bar{B})$ is upper-triangular, for each $k$ we have $\cl_M(\{a_1, \dotsc, a_k\}) = \cl_M(\{b_1, \dotsc, b_k\})$. Let $A_0 = \{a_1, \dotsc, a_s\}$ and $B_0 = \varphi(A_0)$. Let $A' = \bar{A}-A_0$ and $B' = B - B_0$. 
	
	Since $r(M \con A') = s$, by Lemma~\ref{udensity} we have $\tau_{s-1}(M \con A') \le \binom{2s}{s-1} \le 4^s-1$, so by a majority argument there is some $B'' \subseteq B'$ for which $r_{M \con A_1}(B'') \le s-1$ and $|B''| \ge (n-s)/(4^s-1) \ge 4^{-s}n$. Let $B'' = \{b_{i_1}, \dotsc, b_{i_m}\}$ where $1 \le i_1 < \dotsc < i_m \le n$. By Lemma~\ref{goodbasis}, there is some $h$ so that $\{b_{i_1}, \dotsc, b_{i_h}\}$ spans at least $\lfloor \tfrac{m}{s} \rfloor$ elements of $\{b_{i_{h+1}}, \dotsc, b_{i_m}\}$ in $M \con A'$. Let $Y$ be a basis for $\{b_{i_1}, \dotsc, b_{i_h}\}$ in $M \con A'$ and let $B_1 = \cl_{M \con A'}(Y) \cap \{b_{i_{h+1}}, \dotsc, b_{i_m}\}$; thus, $|Y| \le s-1$ and $Y$ is skew to $A'$ in $M$, and $|B_1| \ge \lfloor \tfrac{m}{s} \rfloor \ge \lfloor \tfrac{n}{s4^s} \rfloor \ge (s4^s)^{t-1}$. Let $A_1 = \varphi(B_1)$ and $X = \varphi(Y)$. Set
	\[M' = M \del (B-(B_0 \cup Y \cup B_1)) \con (A - (A_0 \cup X \cup A_1))\]
	Now  $E(M') = (A_0 \cup X \cup A_1) \cup (B_0 \cup Y \cup B_1)$, and by Lemma~\ref{staytriangular}, $(\bar{A},\bar{B}) \cap E(M')$ is upper-triangular in $M'$. Moreover, we have $i < j < k$ for all $a_i \in A_0, a_j \in X, a_k \in A_1$ and similar for $B_0,Y,B_1$. Using upper-triangularity and the fact that $Y$ is a basis for $B_1$ in $M \con A_1$, we have $\cl_{M'}(B_0 \cup Y) = \cl_{M'}(A_0 \cup X)$ and  $B_1 \subseteq \cl_{M'}(A_1 \cup X \cup Y)$. Let $N' = M' \con (X \cup Y)$ and for $i \in \{0,1\}$ let $N_i = N'|(A_i \cup B_i)$. Now
	\begin{align*}
		\lambda_{N'}(A_1 \cup B_1) &= r_{N'}(A_1 \cup B_1) + r_{M'}(A_0 \cup B_0 \cup X \cup Y) - r(M')\\
		&= r_{N'}(A_1) + r_{M'}(A_0 \cup X) - r(M')\\	
		&\le |A_1| + |A_0 \cup X| - r(M')  = 0,
	\end{align*}
	Therefore $N' = N_0 \oplus N_1$. Moreover, since $\sqcap_{M'}(A_1,X \cup Y) = 0$ and $B_1 \subseteq \cl_{M' \con Y}(A_1)$ we have $0 = \sqcap_{M' \con Y}(A_1,X) = \sqcap_{M' \con Y}(A_1 \cup B_1,X)$ and so $N_1 = (M' \con Y \del X)/A_0\del B_0$; by Lemma~\ref{staytriangular} it follows that $(\bar{A},\bar{B}) \cap E(N_1)$ is an upper-triangular pair in $N_1$. Since $r(N_0) = |B_1| \ge (s4^s)^{t-1}$, the inductive hypothesis gives a $(t-1)U_{1,2}$-minor $\hat{N}$ of $N_1$ in which $E(\hat{N}) \cap \bar{A}$ and $E(\hat{N}) \cap \bar{B}$ are bases. Since $\sqcap_{M'}(A_0,X) = 0$ and $\sqcap_{M'}(A_0,B_0) = s = r_{M'}(A_0)$, we have $\sqcap_{M' \con X}(A_0,B_0) = s$ and therefore $\sqcap_{N'}(A_0,B_0) \ge s - |Y| > 0$. Therefore $N_0$ contains a circuit intersecting both $A_0$ and $B_0$, and so $N_0$ has a $U_{1,2}$-minor $N_0$ intersecting $A_0$ and $B_0$. It follows that $N' = N_0 \oplus N_1$ has the required $tU_{1,2}$-minor. 	
\end{proof}

\begin{theorem}\label{selfdual}
	Let $s \ge 2$ and $t \ge 0$. If $A$ and $B$ are disjoint bases of a matroid $M$ with $r(M) \ge 4^{s(s4^s)^t}$, then either
	\begin{itemize}
		\item $M$ has a $U_{s,2s}$-minor $U$ in which $E(U) \cap A$ and $E(U) \cap B$ are bases, or
		\item $M$ has a $tU_{1,2}$-minor $N$ in which $E(N) \cap A$ and $E(N) \cap B$ are bases. 
	\end{itemize}
\end{theorem}
\begin{proof}
	Let $t' = (s4^s)^t$. Assume that the first outcome does not hold. Note that $A$ and $B$ are disjoint bases of $M_1 = (M|(A \cup B))^*$. By Lemma~\ref{triangularone} applied to $M_1$, we see that $M_1$ has a rank-$t'$-minor $M_2$ having a lower-triangular pair $(\bar{A},\bar{B}) \in A^{t'} \times B^{t'}$. Now $M_2^*$ is a rank-$t'$ minor of $M$ and $(\bar{A},\bar{B})$ is an upper-triangular pair of $M_2^*$; the result now follows from Lemma~\ref{triangulartwo}.  
\end{proof}

We can now  prove Theorem~\ref{unavoidable}, which we restate here in a stronger form. 

\begin{theorem}\label{bigmatroid}
	Let $s \ge 0$ be an integer and $M$ be a matroid. Either
	\begin{itemize}
		\item $M$ has a $U_{s,2s}$-minor, 
		\item $M$ has an $s U_{1,2}$-minor, or
		\item $M$ has a minor $N$ so that $|M| - |N| \le 4^{4^{2s^2}}$ and every element of $N$ is a loop or a coloop.
	\end{itemize} 
\end{theorem}
\begin{proof}
	The result is trivial for $s \le 1$. Suppose that $s \ge 2$ and let $h = \tfrac{1}{2}4^{4^{2s^2}}$; note that $h \ge 4^{s(s4^s)^s}$. Let $B$ be a basis for $M$ and let $X = E(M) - B$. If $r_M(X) \ge h$ then $M$ clearly has a rank-$h$ minor with two disjoint bases, and the result follows from Theorem~\ref{selfdual}. Otherwise, let $X'$ be a basis for $M|X$; now $r^*_{M \con X'}(B) = r_M(X') \le h$, so there exists $B' \subseteq B$ so that $|B'| \le h$ and every element of $B-B'$ is a coloop of $N = M \con X' \del B'$. Since $|B'|,|X'| \le h$ and every element of $X-X'$ is a loop of $N$, we have the required minor. 
\end{proof}

\section{Complete Matroids}

Let $a \ge 2$ be an integer. We say a matroid $M$ is \emph{$a$-complete} if $M$ has a basis $B$ such that, for every $I \subseteq B$ with $2 \le |I| \le a$, there is some $e \in E(M)$ for which $I \cup \{e\}$ is a circuit of $M$.  We call such a $B$ a \emph{joint-set} of $M$.  For example, a $2$-complete matroid with joint-set $B$ is the same as a spanning $B$-clique restriction. We will freely use the easily proved fact that, if $M$ is an $a$-complete matroid with joint-set $B$, and $B' \subseteq B$ is a basis of a contraction-minor $M'$ of $M$, then $M'$ is $a$-complete with joint-set $B'$.

Huge $2$-complete matroids do not contain large $3$-complete minors. However, in Lemma~\ref{upgradecomplete} we prove that, for each integer $a>3$, a huge $3$-complete matroid does contain a large $a$-complete minors. Then, in Lemma~\ref{3completewincor}, we prove that a huge $3$-complete matroid has either a large balanced uniform matroid or a large projective geometry as a minor. 
\begin{lemma}\label{buildcomplete}
	Let $m > a \ge 2$ be integers and let $h = \binom{m}{a+1}$. If $M$ is an $a$-complete matroid with joint-set $B$, and there are disjoint sets $B_0, \dotsc, B_h \subseteq B$ so that $|B_0| \ge m$, and for each $i \in \{1,\dotsc, h\}$ we have $|B_i| > a$ and there is some $x_i \in E(M)$ for which $B_i \cup \{x_i\}$ is a circuit, then $M$ has an $(a+1)$-complete minor of rank $m$. 
\end{lemma}
\begin{proof}
	By contracting a subset of $B$, we may assume that $|B_0| = m$, that $|B_i| = a+1$ for each $i \ge 1$, and that $(B_0, \dotsc, B_h)$ is a partition of $B$. Let $\{J_1, \dotsc, J_h\}$ be the $(a+1)$-element subsets of $B_0$. For each $i \ge 1$, let $J_i = \{e_i^1, \dotsc, e_i^{a+1}\}$ and $B_i = \{b_i^1, \dotsc, b_i^{a+1}\}$. Since $M$ is $2$-complete, there exist $f_i^1, \dotsc, f_i^{a+1}$ so that each $\{e_i^j,f_i^j,b_i^j\}$ is a triangle of $M$. Let $F_i = \{f_i^1, \dotsc, f_i^{a+1}\}$. Since $J_i \cup B_i$ is independent, we see that in the matroid $M \con F_i$, the sets $J_i$ and $B_i$ are independent and each $e_i^j$ is parallel to $b_i^j$. It follows that $J_i \cup \{x_i\}$ is a circuit of $M \con F_i$. Moreover, in $M \con B_0$, the set $F_i$ is independent and each $f_i^j$ is parallel to $b_i^j$.  
	
	Let $F = \cup_{i=1}^h F_i$ and let $M_0 = M \con F$. Since each $B_i$ (for $i \ge 1$) is spanned by $J_i$ in $M \con F$, the set $B_0$ spans $M_0$. Moreover, in the matroid $M \con B_0$, the set $B-B_0$ is independent, and each $F_i$ is an independent set spanned by $B_i \subseteq B-B_0$, so $r_{M \con B_0}(F) = r_M(F) = |F|$ and therefore $r_{M \con F}(B_0) = r_M(B_0) = |B_0|$ and $B_0$ is a basis for $M_0$. It follows that $M_0$ is $a$-complete with joint-set $B_0$.  Finally, for each $i \ge 1$, the set $J_i \cup \{x_i\}$ is a circuit in $M \con F_i$ and $J_i$ is independent in $M \con F$, so $J_i \cup \{x_i\}$ is a circuit of $M_0$. Thus, $M_0$ is $(a+1)$-complete with joint-set $B_0$. 
\end{proof}

\begin{lemma}\label{upgradecomplete}
	Let $a,m$ be integers with $m > a \ge 3$. If $M$ is $a$-complete and $r(M) \ge m^{a+1}$, then $M$ has an $(a+1)$-complete minor of rank $m$.
\end{lemma}
\begin{proof}
	Let $B$ be the joint-set of $M$. Note that $r(M) \ge m + (a+2)\binom{m}{a+1}$. Let $h = \binom{m}{a+1}$ and let $B_0,B_1', \dotsc, B_h'$ be disjoint subsets of $B$ so that $|B_0| = m$ and $|B'_i| = a+2$ for $i \ge 1$. For each $i \ge 1$, let $X_i,Y_i$ be subsets of $B'_i$ with $|X_i| = a$, $|Y_i| = 3$ and $|X_i \cap Y_i| = 1$, let $\{z\} = X_i \cap Y_i$, and let $e_i,f_i$ be such that $X_i \cup \{e_i\}$ and $Y_i \cup \{f_i\}$ are circuits of $M$. The matroid $M|((X_i \cup Y_i \cup \{e_i,f_i\}) -\{z_i\})$ is thus the $2$-sum of two circuits at a common element $z_i$, so $(X_i \cup Y_i \cup \{e_i,f_i\}) - \{z_i\}$ is a circuit of $M$. Thus, $B_i = B'_i - \{z_i\}$ is independent and $B_i \cup \{e_i\}$ is a circuit of $M \con f_i$. Let $F = \{f_1, \dotsc, f_h\}$. Since the sets $B_i'$ are mutually skew in $M$ and each $B_i'$ spans $f_i$ in $M$, each $B_i$ is an $(a+1)$-element independent set in $M \con F$ and each $B_i \cup \{e_i\}$ is a circuit of $M \con F$. Clearly $M \con F$ is $a$-complete with joint-set $B - \{z_1, \dotsc, z_h\}$. Thus, the sets $B_0, \dotsc, B_h$ in satisfy the hypotheses of Lemma~\ref{buildcomplete} in $M \con F$, so $M \con F$ has the required minor. 
\end{proof}

\begin{lemma}\label{3completewin}
	Let $s,p \ge 2$ be integers. If $M$ is a $3$-complete matroid with $r(M) \ge (16s^2p^4)^{(2p)^{2p}}$, then $M$ has a minor $N$ with $r(N) > 1$ and $\tau_{s-1}(N) \ge r(N)^p$. 
\end{lemma}
\begin{proof}
	Let $t = 2p$. Let $m_t = s^2t^4$, and for each $i \in \{3,\dotsc, t\}$ recursively set $m_{i-1} = (m_i)^i$. Therefore $m_3 \le (m_t)^{t^t} \le r(M)$. Now $M$ is a $3$-complete matroid of rank at least $m_3$; inductively by choice of the $m_i$ and Lemma~\ref{upgradecomplete}, for each $a \in \{3, \dotsc, t\}$ there is an $a$-complete minor of $M$ of rank at least $m_a$. Setting $a = t$, we see that $M$ has a $t$-complete minor $N$ with $r(N) \ge s^2t^4$. 
		
	Let $B$ be the frame of $N$. For each $t$-element set $I \subseteq B$, let $e_I$ be an element of $N$ such that $I \cup \{e_I\}$ is a circuit. Let $X$ be the set of all such $e_I$; thus $|X| = \binom{r(N)}{t} \ge t^{-t}r(N)^t$. 
	
	We now argue that $\tau_{s-1}(N) \ge r(N)^p$. Suppose that $S \subseteq X$ and $r_N(S) < s$. Let $B_S$ be a basis for $N|S$. Since each $x \in X$ is spanned by a $t$-element subset of $B$, there is a set $J \subseteq B$ with $|J| < st$ so that $J$ spans $B_S$ and thus spans $S$. The set $J$ spans at most $\binom{|J|}{t} < (st)^t$ elements of $X$, so $|S| < (st)^{t}$. Thus, each rank-$(s-1)$ set of $N$ contains at most $(st)^t$ elements of $X$, so \[\tau_{s-1}(N) \ge (st)^{-t}|X| \ge (st^2)^{-t}r(N)^t \ge (r(N))^{-t/2}r(N)^{t} = r(N)^p,\]  as required. 
\end{proof}

Applying the above with $p = f_{\ref{pgdensity}}(s,n)$ gives the following. 

\begin{lemma}\label{3completewincor}
	There is a function $f_{\ref{3completewincor}}\colon \bZ^2 \to \bZ$ so that, if $s,n \ge 2$ are integers and $M$ is a $3$-complete matroid with $r(M) \ge f_{\ref{3completewincor}}(s,n)$, then $M$ has either a $U_{s,2s}$-minor or a rank-$n$ projective geometry minor. 
\end{lemma}

\section{Lifts and Projections}

Let $M$ and $N$ be matroids on a common ground set $E$. 
We let $\dist(M,N)$ denote the minimum $d \ge 0$ such that there is a sequence $M=M_0,\dotsc,M_d =N$ of matroids for which each $M_i$ is an elementary lift or projection of $M_{i-1}$. By duality of projections/lifts we have $\dist(M,N) = \dist(N,M)$, and since the rank zero matroid on $E$ can be obtained from $M$ by a sequence of $r(M)$ elementary projections and likewise for $N$, we have $\dist(M,N) \le r(M) + r(N)$ and thus $\dist(\cdot)$ is always finite. For convenience we set $\dist(M,N) = \infty$ if $E(M) \ne E(N)$. 

For $C \subseteq E(M)$, we write $M \dcon C$ to denote the matroid $(M \con C) \oplus O_C$, where $O_C$ is the rank-zero matroid with ground set $C$. Since $M \dcon C$ is also obtained by extending $M$ by a parallel copy of $C$ and then contracting the new elements, we have $\dist(M,M \dcon C) \le |C|$. We use this fact to bound the distance between two matroids with a common minor; we use the following lemma freely.
\begin{lemma}
	If $M$ and $N$ are matroids and $X$ is a set so that $M \con X = N \con X$ or $M \del X = N \del X$, then $\dist(M,N) \le 2|X|$. 
\end{lemma}
\begin{proof}
	By duality, it suffices to show this for contraction. If $M \con X = N \con X$ then $M \dcon X = N \dcon X$, so  
	$\dist(M,N) \le \dist(M,M \dcon X) + \dist(N \dcon X,N) \le 2|X|.$
\end{proof}

We now consider an operation that turns various elements of a matroid into loops, and moves other elements into parallel with existing elements. Let $M$ be a matroid and let $\phi \notin E(M)$. Let $X \subseteq E(M)$ and let $\psi \colon X \to E(M \del X) \cup \{\phi\}$ be a function. We let $\psi(M)$ denote the matroid with ground set $E(M)$ so that 
\begin{itemize}
	\item $\psi(M) \del X = M \del X$, 
	\item each $x \in X$ for which $\psi(x) = \phi$ is a loop of $\psi(M)$, and 
	\item each $x \in X$ for which $\psi(x) \ne \phi$ is parallel to $\psi(x)$ in $\psi(M)$. 
\end{itemize}

Note that $\si(\psi(M)) \cong \si(\psi(M) \del X) =  \si(M \del X)$. 

Suppose that there exists $C \subseteq E(M)$ such that $\psi(x) = \phi$ for all $x \in X \cap \cl_M(C)$, and for each $x \in X - \cl_M(C)$, the elements $x$ and $\psi(x)$ are parallel in $M \con C$. In this case, we say that the function $\psi$ is a \emph{$C$-shift}, and that $\psi(M)$ is a \emph{$C$-shift} of $M$. 
\begin{lemma}
	If $\wh{M}$ is a $C$-shift of $M$, then $\dist(M,\wh{M}) \le 4|C|$.
\end{lemma}
\begin{proof}
	Let $\psi\colon X \to E(M \del X) \cup \{\phi\}$ be a $C$-shift so that $\wh{M} = \psi(M)$. Let $\psi'$ be the restriction of $\psi$ to the domain $X-C$. Now $\psi'(M)$ is a $C$-shift of $M$. By the definition of a $C$-shift we have $\psi'(M) \con C = M \con C$, and clearly $\psi(M) \del C = \psi'(M) \del C$. Thus $\dist(\psi'(M),M) \le 2|C|$ and $\dist(\psi(M), \psi'(M)) \le 2|C|$, and the lemma follows.  
\end{proof}

The next theorem shows that small perturbations can not introduce arbitarily large projective geometries or balanced uniform matroids. 

\begin{theorem}\label{perturbthm}
	 Let $M$ and $N$ be matroids such that $\dist(M,N) = 1$. 
	\begin{itemize}
		\item If $s \ge 2$ and $M$ has a $U_{s2^{4s},2s2^{4s}}$-minor, then $N$ has a $U_{s,2s}$-minor.
		\item If $n \ge 3$ and $M$ has a $\PG(n-1,q)$-minor, then $N$ has a $\PG(n-3,q)$-minor. 
	\end{itemize}
\end{theorem}
\begin{proof}
	Let $\wh{M}$ be a matroid so that $\{\wh{M} \con e,\wh{M} \del e\} = \{M,N\}$ for some $e \in E(\wh{M})$, so $N$ is an elementary projection of $M$ if and only if $M = \wh{M} \del e$ and $N = \wh{M} \con e$. 
	
	Suppose that $M$ has a $U_{s2^{4s},2s 2^{4s}}$-minor. If $N$ is an elementary projection of $M$, then note that $M$ has a $U_{s+1,s2^{4s}}$-minor $M \con C \del D$. Let $M_0 = \wh{M} \con C \del D$. Note that $M_0 \con e$ is a minor of $N$ and that $s \le r(M_0 \con e) \le r(M_0 \del e) = s+1$. But we also have $M_0 \del e \cong U_{s+1,s2^{4s}}$, so 
	\[ \tau_{s-1}(M_0 \con e) \ge \tau_{s}(M_0) \ge s^{-1}(s2^{4s}) > \tbinom{2s}{s-1}^2 \ge \tbinom{2s}{s-1}^{r(M_0 \con e)-s},\] and so $M_0 \con e$ has a $U_{s,2s}$-minor by Lemma~\ref{udensity}. Thus, $N$ has a $U_{s,2s}$-minor. If $N$ is an elementary lift of $M$, then $N$ has a $U_{s,2s}$-minor by duality. 
	
	Suppose that $M$ has a $\PG(n-1,q)$-minor $G = M \con C \del D$, and as before let $M_0 = \wh{M} \con C \del D$. If $N$ is an elementary projection of $M$ then $M_0 \del e = G \cong \PG(n-1,q)$, so there is a $\PG(n-2,q)$-restriction $R$ of $M_0 \del e$ that does not span $e$ in $M_0$. Thus $(M_0 \con e)|E(R) \cong \PG(n-2,q)$ and so $N$ has a $\PG(n-2,q)$-minor. Suppose that $N$ is an elementary lift of $M$. If there is some line $L$ of $G$ so that $e \in \cl_{M_0}(L)$, then $M_0 \del \{e\} \con L = M_0 \con (L \cup \{e\}) = G \con L$ and $\si(G \con L) \cong \PG(n-3,q)$, so $N $ has a $\PG(n-3,q)$-minor. If every line $L$ of $G$ is skew to $\{e\}$ in $M_0$, then let $B$ be a basis of $G$. Given $x_1,x_2 \in \cl_{M_0}(B)$, the line $\cl_G(\{x_1,x_2\})$ is skew to $\{e\}$ in $M_0$, so $\cl_G(\{x_1,x_2\}) \subseteq \cl_{M_0}(\{x_1,x_2\})$. Thus, $\cl_{M_0}(B)$ contains $B$ and is closed under taking lines of $G$ through two points; since $G \cong \PG(n-1,q)$ it follows that $\cl_{M_0}(B) \supseteq E(G)$. Moreover, $B$ is skew to $\{e\}$ in $M_0$ and so  $M_0|\cl_{M_0}(B) = (M_0 \con e)|\cl_{M_0}(B)  = G$; thus, $N$ has a $\PG(n-1,q)$-minor. 
\end{proof}

\section{Spanning Cliques}

In this section we prove Theorem~\ref{main1}. First we establish a certificate $X$ that guarantees a large projective geometry or balanced uniform minor given a spanning clique. In the representable case where the $B$ is an identity matrix, the set $X$ in the lemma below can be thought of as a set of $m$ columns containing three disjoint $m \times m$ nonsingular matrices. 

\begin{lemma}\label{threenonsingular}
	There is a function $f_{\ref{threenonsingular}}\colon \bZ_2 \to \bZ$ so that for all integers $s,n$ with $s,n \ge 2$, if $m \ge f_{\ref{threenonsingular}}(s,n)$ is an integer, $M$ is a matroid with a spanning $B$-clique restriction, and $B_1,B_2,B_3,X$ are disjoint $m$-element subsets of $E(M)$ with $B_i \subseteq B$, $X \subseteq E-B$, and $r_{M \con (B-B_i)}(X) = m$ for each $i \in \{1,2,3\}$, then $M$ has a $U_{s,2s}$-minor or a rank-$n$ projective geometry minor. 
\end{lemma}
\begin{proof}
	 Given $s,n \ge 2$, let $n_0 = f_{\ref{3completewincor}}(s,n)$. Let $m_3 = n_0 + \binom{n_0}{3}$, and for each $k \in \{0,1,2\}$, recursively set $m_k = 4^{s(s4^s)^{m_{k+1}}}$. Set $f_{\ref{threenonsingular}}(s,t) = m_0$. 
	 
	Let $m \ge f_{\ref{threenonsingular}}(s,t)$ and let $M$ be a $2$-complete matroid with joint-set $B$. Let $B_1,B_2,B_3 \subseteq B$ and $X \subseteq E-B$ be disjoint $m$-element subsets so that $r_{M \con (B-B_i)}(X) = m$ for each $i \in \{1,2,3\}$. Assume that $M$ has no $U_{s,2s}$-minor; we will show that $M$ has a rank-$n$ projective geometry minor. We may assume by contracting $B-(B_1 \cup B_2 \cup B_3)$ that $(B_1,B_2,B_3)$ is a partition of $B$. 
\begin{claim}
There exists an integer $m' \ge m_3$, a rank-$3m'$ contraction-minor $M'$ of $M$, a basis $B' \subseteq B$ of $M'$, a partition $(B_1',B_2',B_3')$ of $B'$, and a set $X' \subseteq X$ for which $M' \con (B'-B'_i) | (B'_i \cup X') \cong m'U_{1,2}$ for each $i \in \{1,2,3\}$. 
\end{claim}
\begin{proof}[Proof of claim:]
	We show for each $k \in \{0,1,2,3\}$ that $M$ has a contraction-minor $M^k$ with basis $B^k \subseteq B$ such that there is a partition $(B^k_1,B^k_2,B^k_3)$ of $B^k$ and a set $X^k \subseteq X$, so that $|B^k_i| = |X^k| \ge m_k$ for each $k$, so that $X^k$ is independent in $M \con (B^k-B^k_j)$ for all $j \in \{k+1,\dotsc,3\}$, and so that $M^k \con (B^k-B^k_j)|(B^k_j \cup X^k) \cong m_kU_{1,2}$ for all $j \in \{1, \dotsc, k\}$. This clearly holds for $k = 0$, and setting $k = 3$ will give the claim. The argument essentially consists of three aplications of Theorem~\ref{selfdual}.
	
	Suppose that these $M^k,B^k,B^k_i$ and $X^k$ exist for some $k \in \{0,1,2\}$. Consider the matroid $P = M^k \con (B^k - B^k_{k+1}) | (B^k_{k+1} \cup X^k)$. The sets $B^k_{k+1}$ and $X^k$ are disjoint bases of $P$, and $r(P) \ge m_k = 4^{(s4^s)^{m_{k+1}}}$ so by Theorem~\ref{selfdual} there are disjoint sets $C,D \subseteq E(P)$ such that $P \con C \del D \cong m_{k+1}U_{1,2}$ and $P \con C \del D$ has both $B^k_{k+1}-C \cup D$ and $X^k - C \cup D$ as bases. Choose $C$ independent and $D$ co-independent in $P$, so $|C| = |D| =  m_k - m_{k+1}$. Let $X^{k+1} = X^k - (C \cup D)$ and $B^{k+1}_{k+1} = B^k_{k+1} - (C \cup D)$. 
	
	For each $j \in \{1,2,3\}$, let $Q^j = (M^k \con C) \con (B^k - B^k_j) | ((B^k_j-D) \cup X^{k+1})$. For $j \le k$, the matroid $Q^j$ is a minor of $M^k \con (B^k - B^k_j) | (B^k_j \cup X^{k}) \cong m_kU_{1,2}$ obtained by removing $(m_k - m_{k+1})$ elements of $X^k$. If $B^{k+1}_j$ is defined to be the set of the $m_{k+1}$ elements of $B^k_j$ that are not parallel to any of these removed elements, then we have $Q^j \con (B^k_j - B^{k+1}_j) \cong m_{k+1}U_{1,2}$. For $j = k+1$ we have $Q^j \cong m_{k+1}U_{1,2}$ by definition. For $j > k+1$, we have $C \subseteq (B^k-B^k_j) \cup (X^k-X^{k+1})$, and $X^k$ is independent in $M^k \con (B^k - B^k_j)$, so $X_{k+1}$ is independent in $Q^j$. Thus, there is an $m_{k+1}$-element set $B^{k+1}_j \subseteq B^k_j$ for which both $B^{k+1}_j$ and $X^{k+1}$ are independent in $Q^j \con (B^k_j - B^{k+1}_j)$. Now setting $B^{k+1} = \cup_{i=1}^3 B^{k+1}_i$ and $M^{k+1} = (M^k \con C) \con (B^k - B^{k+1})$, together with the $B^{k+1}_j$ and $X^{k+1}$ as defined, gives the claim.
\end{proof}

Note, since $M'$ is a contraction-minor of $M$ and $B' \subseteq B$, that the matroid $M'$ is $2$-complete with joint-set $B'$. Now, for each $i \in \{1,2,3\}$ we have $M' \con (B'-B_i')|(B_i' \cup X') \cong m'U_{1,2}$ and $B_i'$ is a basis of $M \con (B-B_i')$, so for each $x \in X'$, there exists a unique $b_i(x) \in B_i'$ such that $\{x,b_i(x)\}$ is a parallel pair of $M \con (B'-B_i')$. Let $B(x) = \{b_1(x),b_2(x),b_3(x)\}$.  It follows for each $x \in X'$ that $B(x) \cup \{x\}$ is a circuit of $M'$. Moreover, since $b_i(x) \ne b_i(y)$ for $x \ne y$, the sets in $\cB = \{B(x): x \in X'\}$ are pairwise disjoint. Let $h = \binom{n_0}{3}$. Since $r(M') = 3m' \ge n_0 + 3h$, there exist $x_1, \dotsc, x_h \in X'$ and a set $B_0 \subseteq B'$ with $|B_0| = n_0$ such that $B(x_1), \dotsc, B(x_h)$ are disjoint from $B_0$. By Lemma~\ref{buildcomplete} with $a = 2$, the matroid $M'$ has a $3$-complete minor of rank $n_0$. The lemma now follows from the definition of $n_0$ and Lemma~\ref{3completewincor}. 
\end{proof}

We now prove a result that will imply Theorem~\ref{main1}.

\begin{theorem}\label{maintech}
	Let $s,n \ge 2$ be integers and let $h =  f_{\ref{threenonsingular}}(s,n)$. If $M$ is a matroid with a spanning $B$-clique restriction having no $U_{s,2s}$-minor and no rank-$n$ projective geometry minor, then there is a set $\wh{B} \subseteq B$ and a $\wh{B}$-clique $\wh{M}$ such that  $\dist(M,\wh{M}) \le 7 h$. Moreover, there are disjoint sets $C_1,C_2 \subseteq E(M)$ satisfying $|C_1| \le 3h$ and  $|C_2| \le h$ such that $\wh{M}$ is a $C_2$-shift of $M \dcon C_1$.  
\end{theorem}
\begin{proof}
	Let $E = E(M)$ and let $X \subseteq E$ be maximal so that there exist disjoint sets $B_1,B_2,B_3 \subseteq B$ for which $|B_i| = |X|$ and $X$ is independent in $M \con (B-B_i)$ for each $i \in \{1,2,3\}$. Since $M$ is $2$-complete with joint-set $B$, Lemma~\ref{threenonsingular} gives $|X| \le h$. Let $\bar{B} = B_1 \cup B_2 \cup B_3$ and $\wh{B} = B - \bar{B}$;  note that $\wh{B}$ is a basis of $M \con (X \cup (\bar{B} - B_i))$ for all $i,j \in \{1,2,3\}$. For each $f \in (E(M) - B \cup X)$, let $H_i(f) \subseteq \wh{B}$ be such that $H_i(f) \cup \{f\}$ is the fundamental circuit of $f$ with respect to the basis $\wh{B}$ of $M \con (X \cup (\bar{B}-B_i))$.  
	\begin{claim}
		There is a partition $(W_0,W_1,W_2)$ of $E - (B \cup X)$ so that 
		\begin{itemize}
			\item $W_0 \subseteq \cl_M(\bar{B} \cup X)$, 
			\item each $f \in W_1$ is parallel to some element of $\wh{B}$ in $M \con (\bar{B} \cup X)$, and
			\item $M \con (B_2 \cup B_3 \cup X) | (\wh{B} \cup W_2)$ is a $\wh{B}$-clique.		
		\end{itemize}
	\end{claim}
	\begin{proof}[Proof of claim:]
		Let $f \in E(M) - (B \cup X)$. Note that if $\hat{b} \in H_i(f)$ then the set $X \cup \{f\}$ is independent in $M \con (B - (B_i \cup \{\hat{b}\}))$. Thus, if $b_1,b_2,b_3 \in \wh{B}$ are distinct and $b_i \in H_i(f)$ for each $i$, then $X \cup \{f\}$ and the three sets $B_i \cup \{b_i\}$ yield a contradiction to the maximality of $X$. Therefore the bipartite graph with bipartition $(\{1,2,3\},\wh{B})$ and edge-set $\{(i,\wh{B})\colon \wh{B} \in H_i(f)\}$ has no three-edge matching. By Hall's theorem, we either have $|H_i(f)| \le 1$ for some $i$, or $|H_1(f)| = 2$.
		
		Let $W_2 = \{f \in E-(B \cup X)\colon |H_1(f)| = 2\}$ and $M' =M \con (B_2 \cup B_3 \cup X) |(\wh{B} \cup W_2)$.  By construction, every $f \in W_2$ is spanned by the two-element set $H_1(f) \subseteq \wh{B}$ in $M'$, so $M'$ is framed by $\wh{B}$. Moreover, since $M$ is a $B$-clique, for every two-element set $I \subseteq \wh{B}$, there is a circuit $I \cup \{f\}$ of $M$; clearly $H_1(f) = I$ and so $f \in W_2 \subseteq E(M')$ and $I$ spans $f$ in $M'$. It follows that $M'$ is a $\wh{B}$-clique. 
						
		If $|H_i(f)| \le 1$ for some $i$, then $f$ is a either a loop or is parallel to some element of $\wh{B}$ in $M \con (X \cup (\bar{B}-B_i))$, so the same holds in $M \con (\bar{B} \cup X)$. Thus, we can set $f \in W_0$ or $f \in W_1$ as appropriate; the claim follows. 
	\end{proof}
	
	Let $C_1 = B_2 \cup B_3 \cup X$ and $N = M \dcon C_1$. Note that $N|(C_1 \cup \wh{B} \cup W_2)$ is a $\wh B$-clique. Define a function $\psi\colon B_1 \cup W_0 \cup W_1 \to E \cup \{\phi\}$ so that $\psi(W_0 \cup B_1) = \{\phi\}$ and so that $\psi(x) \in \wh{B}$ and $\psi(x)$ is parallel to $x$ in $N \con B_1$ for each $x \in W_1$. It is clear that $\psi$ is a $B_1$-shift in $N$. Now $\psi(N) | (C_1 \cup \wh{B} \cup W_2) = N|(C_1 \cup \wh{B} \cup W_2)$ and each $x \in B_1 \cup W_0 \cup W_1$ is a loop or is parallel to some element of $\wh{B}$, so $\psi(N)$ is a $\wh{B}$-clique. Now, since $M \del C_1 = N \del C_1$ and $\psi$ is a $B_1$-shift in $N$, we have 
	\[\dist(M,\psi(N)) \le \dist(M,N) + \dist(N,\psi(N)) \le |C_1| + 4|B_1| \le 7h.\]
	Thus $\wh{M} = \psi(N)$, together with the sets $C_1$ and $C_2 = B_1$, satisfies the theorem.
\end{proof}

To derive Theorem~\ref{main1} from the above result, we can $k(s,n) = 13h$, where $h = f_{\ref{threenonsingular}}(s,n)$, and set $N = \wh{M} \oplus J$, where $J$ is the free matroid with ground set $B - \wh{B}$. It is clear that this $N$ is framed by $B$ and we have $\dist(N,M) \le 2(|B-\wh{B}|) + \dist(\wh{M},M) \le 13h$. 

\section{Spanning Geometries}

	In this section we prove Theorem~\ref{maingeom}, a stronger version of Theorem~\ref{main2}.

	We first prove a lemma regarding a special type of matroid that is `far' from being representable over a field. Let $q$ be a prime power and $h,t \ge 0$ be integers. An \emph{$(q,h,t)$-stack} is a matroid $S$ so that there is a partition $(X_1,\dotsc,X_h)$ of $E(S)$ such that, for each $1 \le i \le h$, the matroid $S \con (X_1 \cup \dotsc \cup X_{i-1}) |X_i$ has rank at most $t$ and is not $\GF(q)$-representable. The following lemma, proved in [\ref{covering2}, Lemma 6.3], shows that a stack on top of a projective geometry yields an interesting minor.  
	
	\begin{lemma}\label{stackwin}
		There is a function $f_{\ref{stackwin}}\colon \bZ^4 \to \bZ$ so that, for every prime power $q$ and for all integers $s,n,t \ge 2$, if $M$ is a matroid with a $\PG(r(M)-1,q)$-restriction and a $(q,f_{\ref{stackwin}}(s,n,q,t),t)$-stack restriction, then $M$ has a $U_{s,2s}$-minor or a $\PG(n-1,q')$-minor for some $q' > q$. 
	\end{lemma}
	
	The proof of the next lemma uses the fact that $U_{s,2s}$-is $\GF(q)$-representable whenever $q \ge 2s$, which holds because the Vandermonde matrix $A \in \GF(q)^{[s] \times \GF(q)}$ defined by $A_{j,\alpha} = \alpha^{j-1}$ is a $\GF(q)$-representation of $U_{s,q}$; see Section 6.5 of [\ref{oxley}] for more detail. 
\begin{lemma}\label{localrep}
	There is a function $f_{\ref{localrep}}\colon \bZ^3 \to \bZ$ so that, for all integers $s,t,n \ge 2$, if $q$ is a prime power, and $M$ is a matroid with a $\PG(r(M)-1,q)$-restriction $R$, no $U_{s,2s}$-minor, and no $\PG(n-1,q')$-minor for any $q' > q$, then there is a set $C \subseteq E(M)$ so that $|C| \le f_{\ref{localrep}}(s,n,t)$, every nonloop of $M \con C$ is parallel to an element of $R$, and every rank-$t$ restriction of $M\con C$ is $\GF(q)$-representable.
\end{lemma}
\begin{proof}
	Let $Q$ be the set of all prime powers less than $2s$. Given integers $s,t,n \ge 1$, let $t' = \max(s,t)$. Let $k_1 = \max_{q \in Q}f_{\ref{exppgdensity}}(s,q,n)$ and let $k_2 = \max_{q \in Q}f_{\ref{stackwin}}(s,n,q,t')$. Set $f_{\ref{localrep}}(s,n,t) = k = k_1 + s + t'k_2$.
	
	Let $q$ be a prime power and let $M$ be a matroid with a $\PG(r(M)-1,q)$-restriction. Suppose that $M$ has no $U_{s,2s}$-minor and no $\PG(n-1,q')$-minor for $q' > q$. If $r(M) < s $ then the theorem clearly holds with $C$ equal to a basis for $M$, so we may assume that $r(M) \ge s$. If $q \ge 2s$ then $R$ has a $U_{s,2s}$-restriction, so we may also assume that $q \in Q$. 
	
	Let $F \subseteq E(M)$ be maximal so that $F$ is skew to every rank-$(s-1)$ flat of $R$. Since every rank-$(s-1)$ set of $(M \con F)|E(R)$ is also a rank-$(s-1)$ set of $R$ and thus contains at most $\tfrac{q^{s-1}-1}{q-1}$ elements of $|R|$, we have 
		\[ \tau_{s-1}((M \con F)|E(R)) \ge |R| / \left(\tfrac{q^{s-1}-1}{q-1}\right) > q^{r(M)-s} = q^{r(M \con F) + r_M(F) -s}.\]
		If $r_M(F) \ge k_1+s$ then, since $k_1 \ge f_{\ref{exppgdensity}}(s,q,n)$, we obtain a contradiction from Theorem~\ref{exppgdensity}. Therefore $r_M(F) < k_1 + s$. By the maximality of $F$, every $f \in E(M \con F)$ is spanned by some rank-$(s-1)$ set of $(M \con F)|E(R)$. Let $h$ be maximal so that $M \con F$ has a $(q,h,t')$-stack restriction $S$. By Lemma~\ref{stackwin}, we have $h \le k_2$ and so $r(S) \le t' k_2$. Let $M'= M \con (F \cup E(S))$. By the maximality of $h$, every restriction of $M'$ with rank at most $t'$ is $\GF(q)$-representable. In particular, every rank-$(s-1)$ restriction of $M'$ is $\GF(q)$-representable and every nonloop $f$ of $M'$ is spanned by a $\PG(s-2,q)$-restriction $P_f$ of $M'$ contained in $E(R)$, so every such $f$ is in fact parallel to some element of $P_f$, as otherwise $M'|(E(P_f) \cup \{f\})$ is not $\GF(q)$-representable. Now $r_M(F \cup E(S)) \le r_M(F) + r_{M \con F}(E(S)) \le k_1 + s + t'k_2$, so any basis $C$ for $M|(F \cup E(S))$ will satisfy the lemma. 
\end{proof}

We now state and prove a stronger version of Theorem~\ref{main2}. 

\begin{theorem}\label{maingeom}
	Let $s\ge 2$ and $n \ge 2$ be integers and let $q$ be a prime power. Let $k = f_{\ref{localrep}}(s,n,s-1)$. If $M$ is a rank-$r$ matroid with a $\PG(r-1,q)$-restriction, no $U_{s,2s}$-minor, and no $\PG(n-1,q')$-minor for any $q' > q$, then there is a matroid $\wh{M}$ such that $\si(\wh{M}) \cong \PG(r-1,q)$ and $\dist(M,\wh{M}) \le 4k$. Moreover, there is a set $C \subseteq E(M)$ for which $|C| \le k$ and $\wh{M}$ is a $C$-shift of $M$. 
\end{theorem}
\begin{proof}
	Let $R$ be a $\PG(r-1,q)$-restriction of $M$. By Lemma~\ref{localrep}, there exists a set $C \subseteq E(M)$ so that $|C| \le k$ and every nonloop of $M \con C$ is parallel to an element of $E(R)$. Let $\psi\colon E(M) - E(R) \to E(R) \cup \{\phi\}$ be defined so that $\psi(\cl_M(C) - E(R)) = \{\phi\}$ and each $x \in E(M) - (E(R) \cup \cl_M(C))$ is parallel to $\psi(x)$ in $M \con C$. Clearly $\psi$ is a $C$-shift and $\si(\psi(M)) \cong R$; the theorem follows. 
\end{proof}

	\section{Bicircular Cliques}\label{bicircsection}

	We now prove Theorem~\ref{bicirc}, first showing that a spanning framed bicircular clique restriction together with a spanning projective geometry restriction gives a large uniform minor. We implicitly use the well-known fact that $B^+(H)$ is a minor of $B^+(G)$ whenever $H$ is a minor of $G$. 
	
	\begin{lemma}\label{bicircpg}
		There is a function $f_{\ref{bicircpg}}\colon \bZ \to \bZ$ so that, for every integer $s \ge 2$, if $r \ge f_{\ref{bicircpg}}(s)$ and $M$ is a rank-$r$ matroid with a $B^+(K_r)$-restriction and a rank-$r$ projective geometry restriction, then $M$ has a $U_{s,2s}$-minor. 
	\end{lemma}
	\begin{proof}
		Let $s \ge 2$ be an integer. Let $q_h$ be the smallest prime power at least $2s$, and let $q_1, \dotsc, q_h$ be the prime powers at most $q_h$ listed in increasing order. Define $r_1, \dotsc, r_h$ recursively by $r_h = s$, and $r_i = 2s(f_{\ref{localrep}}(s,2,r_{i+1}) + 2)$ for each $i \in \{1, \dotsc, s-1\}$. Set $f_{\ref{bicircpg}}(s) = \max_{1 \le i \le s}r_i$.  
		
		Let $M$ be a matroid of rank $r \ge f_{\ref{bicircpg}}(s)$ that has a $B^+(K_r)$-restriction and a rank-$r$ projective geometry restriction. Assume that $M$ has no $U_{s,2s}$-minor. Let $i \in \{1, \dotsc, h\}$ be maximal so that $M$ has a minor $N$ with rank $n \ge r_i$, so that $N$ has both a $B^+(K_n)$-restriction and a $\PG(n-1,q)$-restriction for some $q \ge q_i$. (This holds for $i = 1$, so this choice is well-defined.) If $q \ge 2s$ then, since $r_i \ge s$ and $M$ has a $\PG(r_i-1,q)$ restriction, $M$ has a $U_{s,2s}$-minor.  Therefore $q < 2s$ and, in particular, $i < h$. Let $n' = f_{\ref{localrep}}(s,2,r_{i+1})+2$ and let $K$ denote the loopless graph on $n'$ vertices in which each pair of vertices is joined by exactly $2s$ distinct edges. Since $n \ge r_i \ge 2s n'$, the graph $K$ is a minor of $K_n$, so $N$ has a rank-$n'$ contraction-minor $N'$ with $B^+(K)$ as a restriction; clearly $N'$ also has a $\PG(n'-1,q)$-restriction. 
		
		By the maximality of $i$, the matroid $N'$ has no $\PG(r_{i+1}-1,q')$-minor for any $q' > q$. By Lemma~\ref{localrep}, there thus exists $C \subseteq E(N)$ so that $|C| \le n'-2$ and each rank-$2$ restriction of $N' \con C$ is $\GF(q)$-representable. As $|V(K)| \ge |C|+2$, there are distinct $v_1,v_2 \in V(K)$ so that $\{v_1,v_2\}$ is skew to $C$ in $N'$. Now $v_1$ and $v_2$ span a line of $B^+(K)$ containing $2s + 2 > q+2$ points of $N' \con C$, so $N' \con C$ has a $U_{2,q+2}$-restriction, contradicting the choice of $C$.	
	\end{proof}
	
	Finally we restate and prove Theorem~\ref{bicirc}. 
	
	\begin{theorem}\label{bicirctech}
		There is a function $f_{\ref{bicirctech}}\colon \bZ \to \bZ$ so that, for every integer $s \ge 2$, if $M$ is a matroid without a $U_{s,2s}$-minor and with a $B^+(K_{r(M)})$-restriction framed by $B$, then there is a set $\wh{B} \subseteq B$ and a $\wh{B}$-clique $\wh{M}$ such that $\dist(M,\wh{M}) \le  f_{\ref{bicirctech}}(s)$.
	\end{theorem}
	\begin{proof}
		For each $s\ge 2$ set $f_{\ref{bicirctech}}(s) = n = 7f_{\ref{threenonsingular}}(s,f_{\ref{bicircpg}}(s))$. Let $M$ be a matroid with a $B^+(K_{r(M)})$-restriction  framed by $B$ and with no $U_{s,2s}$-minor. By Theorem~\ref{maintech}, either $M$ has a rank-$f_{\ref{bicircpg}}(s)$ projective geometry minor $N$, or there is a set $\wh{B} \subseteq B$ and a $\wh{B}$-clique $\wh{M}$ with $\dist(M,\wh{M}) \le n$. In the second case the theorem is immediate. In the first case, let $C$ be such that $N$ is a spanning restriction of $M \con C$. Now $M \con C$ also has a $B^+(K_{r(M \con C)})$-restriction, and the result follows from Lemma~\ref{bicircpg}.
	\end{proof}

\section*{References}
\newcounter{refs}

		\begin{list}{[\arabic{refs}]}
{\usecounter{refs}\setlength{\leftmargin}{10mm}\setlength{\itemsep}{0mm}}


\item\label{highlyconnected}
J. Geelen, B. Gerards, G. Whittle,
The highly-connected matroids in minor-closed classes,
Ann. Comb. 19 (2015), 107-123.

\item\label{gkep}
J. Geelen, K. Kabell, 
The {E}rd{\H o}s-{P}\'osa property for matroid circuits,
J. Combin. Theory Ser. B 99 (2009), 407--419.  

\item\label{covering1}
J. Geelen, P. Nelson, 
Projective geometries in exponentially dense matroids. I, 
J. Combin. Theory Ser. B 113 (2015), 185--207.

\item\label{gr}
C. Godsil and G. Royle, 
Algebraic Graph Theory, 
Springer, 2001. 

\item\label{covering2}
P. Nelson, 
Projective geometries in exponentially dense matroids. II, 
J. Combin. Theory Ser. B 113 (2015), 208--219.

\item\label{margulis}
G. A. Margulis, 
Explicit constructions of graphs without short cycles and low density codes, 
Combinatorica 2 (1982), 71--78. 

\item \label{oxley}
J. G. Oxley, 
Matroid Theory,
Oxford University Press, New York (2011).
\end{list}

\end{document}